\documentclass{amsart}
\usepackage{graphicx} 

\title{Knotting Minimal Sets}

\usepackage{amssymb,amsmath,amsrefs, caption}
\usepackage{mathtools} 
\usepackage[colorlinks=true]{hyperref}
\hypersetup{urlcolor=blue, citecolor=green}
\usepackage{color}
\usepackage{comment}

\usepackage{soul,color}
\usepackage{graphicx}
\usepackage{amssymb}
\usepackage{epstopdf}
\usepackage{nicefrac}
\usepackage{enumerate}
\usepackage{float}

\usepackage{hyperref}
\usepackage{tikz}
\usetikzlibrary{arrows}
\usepackage{amssymb}
\usepackage{epstopdf}
\usepackage{nicefrac}
\usepackage{enumerate}
\usepackage{float}
\usepackage{tikz-cd}

\usepackage{hyperref}
\usepackage{pst-node}
\usepackage{tikz-cd} 
\usetikzlibrary{arrows}

\newcommand{\mS}{{\mathbb S}}

\newcommand{\R}{{\mathbb R}}

\usepackage[all,cmtip]{xy}

\newtheorem{thm}{thm}[section]
\newtheorem{theorem}[thm]{Theorem}

\newtheorem{cor}[thm]{Corollary}
\newtheorem{defn}[thm]{Definition}
\newtheorem{ex}[thm]{Example}

\newtheorem{lemma}[thm]{Lemma}

\newtheorem{prop}[thm]{Proposition}

\newtheorem{remark}[thm]{Remark}
\newtheorem{aside}[thm]{Aside}

\newcommand{\Z}{\mathbb{Z}}

\newcommand{\eX}{{\underleftarrow\lim(
X_i,f_i,p_i)}}
\newcommand{\eY}{{\underleftarrow\lim(Y_i,g_i,q_i)}}

\author{Alex Clark}
\address[Alex Clark]{Centre for Complex Systems, Queen Mary University of London, London, E1 4NS, UK}
\email{alex.clark@qmul.ac.uk}

\author{John Hunton}
\address[John Hunton]{Department of Mathematical Sciences, Durham University, Upper Mountjoy Campus, Stockton Road, Durham, DH1 3LE, UK}
\email{john.hunton@durham.ac.uk}

\date{\today}

\begin{document}

\begin{abstract}
We consider the ways minimal sets of flows in $S^3$ may be embedded. We prove that given any $C^2$ flow on $S^3$ with positive entropy, there is an uncountable collection $\mathcal{M}$ of topologically distinct minimal sets  such that for each $M\in \mathcal{M}$ there are infinitely many embedded copies of $M$ in the flow, each copy with a distinct knot type, thus extending work of Franks \& Williams for periodic orbits.
\end{abstract}

\maketitle

\addtocontents{toc}{\protect\thispagestyle{empty}}
\tableofcontents
\thispagestyle{empty}

\setcounter{page}{1}

\section{Introduction}

In this article we consider the embeddings of metric, compact, one-dimensional minimal sets of flows in the 3-sphere $S^3$. 

Let us recall some basic notions behind this statement. A \emph{flow} is a continuous action of $(\mathbb{R},+)$, and a \emph{minimal set}  $M$ of a flow  is a closed invariant set which contains no proper, non-empty, closed invariant subset. Equivalently, $M$ is minimal if the flow orbit of each point of $M$ is dense in $M$. We shall hereafter refer to any such set simply as a \emph{minimal set} for brevity. 

The simplest one-dimensional minimal sets are periodic orbits, i.e., homeomorphic images of the circle $S^1$, whose embeddings in $S^3$ are covered by knot theory. The study of the various knots that occur as periodic orbits of flows has been an area of significant study by many authors, see, for example, \cite{BW1},\cite{BW2},\cite{FW},\cite{G}. However, minimal sets in general typically have a significantly richer structure, with all but the periodic examples being locally the product of a Cantor set with a one dimensional arc. Indeed, a minimal set is connected, but apart from the periodic orbits, one-dimensional minimal sets are not path connected. By fundamental dimension theory, any one-dimensional compact metric space such as our minimal sets can be embedded in $S^3$. At the same time, for any embedded minimal set $M \hookrightarrow S^3$, as with any one-dimensional compact space, $S^3 \setminus M$ is not separated, and hence is connected and (locally) path connected. Our study here is to engage with the study of the possible embeddings of such examples in $S^3$, both as minimal sets as flows, and more generally as just embeddings as subspaces. 

One class of  minimal sets are the solenoids, which have been studied intensively in the way they can occur as attractors in $S^3$ and other compact 3 manifolds. See, for example \cite{BSG}, \cite{JNW} and the references therein. While there is some relation between that work and the current paper, we address here rather different issues, considering general minimal sets and their embeddings.

While the typical minimal set is far from a simple knotted circle, some ideas from knot theory can usefully be adapted. Firstly, given any space $X$, two embeddings $e_{1,2}\colon X \hookrightarrow S^3$ will be considered \emph{equivalent} if there is an orientation-preserving homeomorphism $h\colon S^3 \to S^3$ satisfying $h\circ e_1 =e_2$. Such an $h$ is isotopic to the identity map on $S^3$, and so there is an ``ambient isotopy'' of equivalent embeddings. The goal of classical knot theory is to study knots, embeddings of $S^1$ in $S^3$, up to ambient isotopy; we shall consider the possible embeddings of minimal sets up to this notion of equivalence. Note that we are interested in the global embedding, or knottedness, of the entire minimal set in $S^3$, and not with issues concerning the linkages of individual orbits, as studied, for example, in \cite{GST} for solenoids; that is a rather different matter altogether.

Our work utilises the recent results of the authors \cite{CHcom} which gave a complete invariant of oriented minimal sets up to homeomorphism. This is clearly much weaker than embeddings of them up to equivalence, but the feature of \cite{CHcom} that we utilise is a presentation of a minimal set in terms of inverse limits of one point unions of circles. We extend this in section \ref{setup} to give a notion of a \emph{presentation} of a minimal set in $S^3$ in terms of a nested sequence of 3-dimensional handlebodies in $S^3$. This has a parallel with the work of Barge and S\'anchez-Gabites \cite{BSG} who represent their solenoidal attractors as limits of solid tori. Moreover, our presentations also allow a natural notion of when a minimal set in $S^3$ is \emph{unknotted}.

In section \ref{knotgroup} we consider the \emph{knot group} of an embedded minimal set, namely the fundamental group of the complement of the set,  a concept directly analogous to the corresponding idea in traditional knot theory, and used in \cite{BSG} in the case of their solenoids. We prove an analogue of the classical result, that the knot group of a knotted $S^1$ in $S^3$ is $\Z$ if and only if it is the unknot. In our case, we have Theorem \ref{unknotfreelimit} that characterises unknotted minimal sets in terms of direct systems of free groups. Also as in the classical theory, the homology of the complement carries more limited information: for a knot, an embedded copy of $S^1$ in $S^3$, the homology of its complement is always $\Z$. We show, Theorem \ref{AlexanderDuality}, that the homology of the complement of a minimal set $M$ is isomorphic to the Cech cohomology of $M$ itself, so it sees the minimal set, but not the embedding.

In section \ref{Suspensions} we recall the important example of $\Sigma(n)$, the suspension of the full shift on $n$ letters, and prove

\medskip\noindent{\bf Theorem \ref{density}.} \emph{Let $M$ be any minimal set of $\Sigma(n)$. The union of all minimal sets of $\Sigma(n)$ homeomorphic with $M$ is a dense subset of $\Sigma(n)$.}

\medskip \noindent
This will be key to our main result (section \ref{templates})

\begin{theorem}\label{ambitious}
Given any $C^2$ flow on $S^3$ with positive entropy on a compact invariant set, there is an uncountable collection $\mathcal{M}$ of topologically distinct minimal sets such that for each $M\in \mathcal{M}$ there are infinitely many embedded copies of $M$ in the flow, each copy with a distinct knot type, i.e., inequivalent up to ambient isotopy.
\end{theorem}

This can be seen as an extension to general minimal sets of the work of Franks and Williams \cite{FW} on periodic orbits. Any given $M \in \mathcal{M}$ can be seen to play the role that periodic orbits do in their work.

The second ingredient for our main result is an analysis of a special class of minimal set embeddings, those with \emph{surface expansions}, section \ref{surfaceExp}. Examples considered in section \ref{excon} illustrate that the computation of the knot group of a general minimal set is formidable, but when an embedded minimal set has a surface expansion, analysis is much more tractable.

Our final ingredient is the notion of \emph{template} as introduced by Birman and Williams \cite{BW1}, \cite{BW2} as a way of describing invariant sets in $S^3$. We recall this in section \ref{templates} where we prove Theorem \ref{ambitious}.

While the main Theorem \ref{ambitious} is clearly in the context of minimal sets of flows, the work on the knot group, Sections \ref{knotgroup} and \ref{excon}, and on surface expansions, Section \ref{surfaceExp}, are applicable to any embedded minimal set in $S^3$ irrespective of whether it is the minimal set of any particular flow.

Finally, let us note that minimal sets play a crucial role in the study of flows and were vital to the initial solution to the Seifert conjecture \cite{S}. The only minimal sets of the $C^1$ non-singular flows on $S^3$ that Schweitzer discovered are isolated, aperiodic, one-dimensional minimal sets. Subsequently, Handel \cite{H} showed that any isolated one-dimensional minimal set of a $C^1$ flow on $S^3$ must be homeomorphic to a minimal set of a flow on a surface, a \emph{surface minimal set} for brevity. Our work allows us to note the extension of this, Remark \ref{Handel}, that any surface minimal set occurs unknotted as an isolated minimal set in a non-singular $C^1$ flow on $S^3$.
Thus, we have a dichotomy between the flows with isolated one-dimensional minimal sets (with zero entropy) in which the known examples occur unknotted, and the positive entropy smooth flows in which the minimal sets occur inside an incredibly complex invariant set resulting from a horseshoe and in which an infinite variety of knotting occurs. It should be noted that the minimal set discovered by Kuperberg \cite{K} has zero entropy, as determined in \cite{HR}.

\section{Flow expansions, presentations of minimal sets and the notion of unknottedness}\label{setup}

Our starting point in the study of embedded minimal sets $M$ is a result of \cite{CHcom}, Section 3.2, on the presentation of such a space, up to homeomorphism, in terms of inverse limits of one point unions, or \emph{wedges}, of circles. In a wedge of circles, we refer to the point where all circles meet as the \emph{wedge point}.

\begin{theorem}
(\cite{CHcom}) Any one-dimensional minimal set $M$ admits (after a change of time) a presentation as an inverse limit of wedges of oriented circles
\[
M \cong \underleftarrow{\lim}\left(X_1 \xleftarrow{f_1} X_2 \xleftarrow{f_2} X_3 \xleftarrow{f_3} \cdots \right)
\]
in which the bonding maps $f_i\colon X_{i+1} \to X_i$ are positive maps preserving the wedge points and the projections $p_i\colon M \to X_i$ are factor maps.
\end{theorem}

In this statement the wedges of circles $X_i$ admit \emph{partial} flows which are factors (after a possible time change) of the flow on the underlying minimal set, and the bonding maps $X_{i+1}\to X_i$ are covering maps away from the wedge point and factor the partial flows.  
We refer the reader to \cite{CHcom} for a detailed treatment. 

The number of circles $n_i$ used to form $X_i$ is finite but may not be bounded. We refer to any presentation as above as a \emph{flow expansion}. 

The reader will note a classical analogue and example of this in the presentation of the one dimensional solenoids. For the solenoids we only need each $X_i$ to be a single circle. In their standard presentations the bonding maps are covering maps which factor the periodic flows on the circles, and which also are factors of the natural flow on the underlying solenoid. Presentations using only single circles will however not go beyond the solenoid examples.

Note that the data in such an inverse system, namely the number of circles in each wedge sum $X_i$ and the bonding maps $f_i\colon X_{i+1}\to X_i$, only determines the homeomorphism type of $M$, not the ambient isotopy class of any specific embedding. We now address the issue of incorporating this further information. We begin with a very general definition.

\begin{defn}
    A \emph{neighbourhood presentation} of a minimal set $M\subset S^3$ is a sequence of nested closed subspaces $S_i\subset S^3$
\[
S^3 \hookleftarrow S_1 \xhookleftarrow{e_1} S_2\xhookleftarrow{e_2} S_3  \cdots
\]
in which each $S_i$ has non-empty interior, $S_{i+1}$ is embedded within the interior of $S_i$ and $M=\cap_i S_i$. 
\end{defn}

Such neighbourhood presentations certainly exist: given an embedded minimal set $M\hookrightarrow S^3$ together with a flow expansion $\eX$ 
for $M$, extend each of the projections $p_i \colon M \to X_i$ to a neighbourhood of $M$ in $S^3$ (recall that a wedge of circles is an absolute neighbourhood retract (ANR)). This allows one to identify a nested sequence of regular neighbourhoods $S_i$ associated to the $X_i$ with $\cap_i S_i$ the embedded copy of $M$ and with the projection $p_i$ extending to $S_i$. However, in the general case there is no reason to suppose much about the structure of the $S_i$'s. As in the case of knots, in order to avoid wild behaviour one needs to require some form of smoothness and in our setting it is natural to consider embeddings of $M$ within non-singular $C^1$ flows on $S^3$, in which case the  neighbourhoods $S_i$ can be taken to be handlebodies, as described below.

Recall that a \emph{handlebody} is a closed, oriented, regular neighbourhood of a finite, connected, oriented graph in $S^3$. Equivalently, it is a closed 3-manifold constructed by gluing a finite number of `handles' (spaces homeomorphic to the product of a closed 2-disc and an interval, $D^2\times I$) by their ends $D^2\times\{0,1\}$ to the boundary of a closed 3-ball, $D^3$. The positioning of the ends of the handles on the 3-ball does not affect its homeomorphism class, and the isotopy class of $S\subset S^3$ is unchanged as we move these positions around. The boundary of a handlebody is a closed surface, and the \emph{genus} of a handlebody is the genus of its surface boundary. As all our handlebodies and associated graphs from now on will be oriented, we shall drop mentioning this from the descriptions.

\begin{defn}
    A \emph{handlebody presentation} of a minimal set $M\subset S^3$ is a neighbourhood presentation $\{S_i\subset S^3\}$ in which each $S_i$ is a handlebody and if $q_i\colon S_i\to G_i$ is the projection of $S_i$, the regular neighbourhood of a finite connected graph $G_i$, to its underlying graph, then the  map $G_{i+1}\to G_i$ induced by the inclusion $S_{i+1}\subset S_i$ preserves orientation and is onto. 
    
    If $\eX$ is a flow expansion for $M$, an \emph{$\{X_i\}$-handlebody presentation} is a handlebody presentation for which there is a commutative diagram
\[
\xymatrix{
\cdots \,\ar@{^{(}->}[r] & \ar[d]\ar@{^{(}->}[r] S_{i}
&\cdots \ar@{^{(}->}[r]
& S_{3}\ar[d]\ar@{^{(}->}[r]_{e_{2}}
& S_{2}\ar@{^{(}->}[r]_{e_{1}}\ar[d]
& S_{1}
\ar[d]
\\
\cdots\ar@{->>}[r]
& X_i\ar@{->>}[r]
&\cdots\ar@{->>}[r]
& X_{3}\ar@{->>}[r]_{\;\;f_{2}}
& X_{2}\ar@{->>}[r]_{\;\;f_{1}}
&X_{1}}
\]
\noindent where the vertical arrows are deformation retractions. 
\end{defn}

As any handlebody $S\subset S^3$ is isotopic to a regular neighbourhood of a wedge of (possibly knotted) circles, any handlebody presentation of a minimal set $M\subset S^3$ is in fact an $\{X_i\}$-handlebody presentation for some flow expansion $\eX$.

In the classical study of knots, there is a standard, trivial, embedding of $S^1$, for example as the unit circle in the $xy$ plane in $\R^3\subset S^3$, which, along with all its equivalent embeddings, is regarded as the \emph{unknot}. Remarkably, given the complexity and flexible nature of a general minimal set $M$, there is an analogue for an \emph{unknotted} embedding of $M$, and we shall see that every such $M$ has such an unknotted embedding. The construction of this embedding is a natural generalisation of the standard embedding of a dyadic solenoid as the intersection of a nested sequence of unknotted solid tori, each one wrapping twice around the torus from the previous stage. This embedding is mirrored by the standard presentation of the dyadic solenoid as the limit of an inverse sequence of circles with bonding maps given by the doubling map of the circle. 

\begin{defn}\label{HBunk}
    Say that a handlebody $S\subset S^3$ is \emph{unknotted} if its embedding is isotopic to the regular neighbourhood of a finite connected graph lying in the $xy$ plane.
\end{defn}

\begin{remark}\label{unkS}{\em 
Note that if the handlebody $S\subset S^3$ is unknotted, then each handle must itself be unknotted (in the classical sense of the unknottedness of the embedded $S^1$ given by the central line $\{{\bf 0}\}\times [0,1]\subset D^2\times [0,1]$, with ends joined by a simple arc along the surface of the 3-ball). In fact, this is sufficient for $S$ to be unknotted: if each handle is unknotted, then any `braids' different handles make with each other may be unwound by simply moving their junctions with the central 3-ball around its surface.
}\end{remark}

It can be convenient to have a `standard model' of an unknotted genus $n$ handlebody in $S^3$. We take this as a regular neighbourhood of the rose of $n$ loops laid out in the $xy$ plane, with wedge point the origin.

\begin{defn}
    Say that a minimal set $M\subset S^3$ is \emph{unknotted} if it has a handlebody expansion $M=\cap S_i$ in which each $S_i$ is unknotted in $S^3$.
\end{defn}

\begin{theorem}
    For every minimal set $M$ there is a homeomorphic copy of $M$ in $S^3$ that is unknotted.
\end{theorem}

\begin{proof}
Given $M$ we can realise its homeomorphism class via a flow expansion $M=\eX$. Begin with the first wedge of circles $X_1$. Embed a regular neighbourhood of $X_1$, say $S_1$, in $S^3$ as an unknotted handlebody according to our `standard model', as above. For each $S^1$ in $X_2$ we may embed a based loop inside $S_1$ so that it is unknotted in $S^3$: this can be done in the same way as the induction step in constructing the standard embedding of one stage of a solenoid in the previous one. Repeat this for each wedge summand of $X_2$ (avoiding the image of previous loops) and take a (small) regular neighbourhood of the resulting embedding of $X_2$ as our $S_2$. As each image loop is unknotted, $S_2$ is unknotted in $S^3$ by Remark \ref{unkS}. The process now iterates to give an unknotted $\{X_i\}$-handlebody presentation of $M$ in $S^3$.
\end{proof}

\begin{prop}\label{unknotpair}
    Suppose $S_{i+1}\subset S_i\subset S^3$ is a nested pair of handlebodies in $S^3$. If $S_{i}$ is knotted, then so is $S_{i+1}$.
\end{prop}

\begin{proof}
    Suppose $S_i$ is knotted. By the Remark \ref{unkS} we know that at least one of its handles, $k$ say, is knotted. By the definition of handlebody presentation, we know there is at least one handle of $S_{i+1}$, say $l$, that passes around $k$. Consider the knot $L$ in $S_{i+1}$ consisting of the handle $l$, joining the ends of the handle as in Remark \ref{unkS}. In turn, consider $S_{i+1}\subset T$, where $T$ is the genus 1 handlebody formed from the union of $S_i$ with a large 3-ball that includes all of $S_i$ except for the handle $k$. Then $T$ is knotted since $k$ is, and $L$ is a satellite knot of $T$, its companion. Hence $L$, and so the handle $l$, is knotted, for example by \cite{Marc}. Thus $S_{i+1}$  is knotted.
\end{proof}

We can use this to see that the notion of a minimal set $M$ in $S^3$ being unknotted does not depend on finding a special (unknotted) presentation.

\begin{cor}\label{indep}
    If $M\subset S^3$ is unknotted, and so has a handlebody presentation $M=\cap S_i$ in which each $S_i$ is unknotted, then every handlebody presentation of $M\subset S^3$ consists of unknotted handlebodies. 
\end{cor}
To prove this we shall recall a result from \cite{CHcom}.

\begin{theorem}\label{unpointed}(\cite{CHcom} Theorem 5.3)
Given any two flow expansions $\eX$ and $\eY$ of the same minimal set $M$, there is a commutative diagram as follows in which the maps $d_i$ and $u_i$ preserve the direction of the flow (but in general may not take wedge point to wedge point).
\[
\begin{tikzcd}[font=\large] 
X_{m_1}  \arrow[swap]{d}{d_1} && X_{m_2}\arrow[swap]{d}{d_2}\arrow{ll}{f_{m_1,\,m_2}}&&X_{m_3}\arrow{ll}{f_{m_2,\,m_3}}\arrow[swap]{d}{d_3} &&  \cdots \arrow{ll}  \\ Y_{n_1} && Y_{n_2} \arrow{ll}{g_{n_1,\,n_2}}\arrow{llu}{u_1} &&Y_{n_3} \arrow{ll}\arrow{ll}{g_{n_2,\,n_3}}\arrow{llu}{u_2}  &&\cdots \arrow{ll}   
\end{tikzcd}
\]
\end{theorem}

\smallskip\noindent\emph{Proof of Corollary} \ref{indep}.
    Suppose given two handlebody presentations of $M\subset S^3$, say $\{S_i\}$ and $\{T_j\}$, corresponding to flow expansions $\eX$ and $\eY$. (For this, the flow structure might need to be adjusted for the two expansions as we cannot assume the two flows on $M$ are necessarily conjugate.) By Theorem \ref{unpointed} we have a commutative zig-zag diagram of maps as above. After the necessary telescoping, and possibly shrinking of cross-sections of the neighbourhoods, we obtain the following commutative diagram in which all the maps are inclusions.
  
\[
\begin{tikzcd}[font=\large] 
S_{m_1}  \arrow[swap]{d} && S_{m_2}\arrow[swap]{d}\arrow{ll}&&S_{m_3}\arrow{ll}\arrow[swap]{d} &&  \cdots \arrow{ll}  \\ T_{n_1} && T_{n_2} \arrow{ll}\arrow{llu} &&T_{n_3} \arrow{ll}\arrow{ll}\arrow{llu}  &&\cdots \arrow{ll}   
\end{tikzcd}
\]
Suppose now that the presentation $M=\cap S_i$ is unknotted. Then Proposition \ref{unknotpair} implies that the $T_j$'s must all be unknotted.
\hfill$\square$

\section{The knot group for minimal sets}\label{knotgroup}

One classical invariant for the embedding of knots is the knot group, the fundamental group of the complement of the knot. As we shall see, this group can also play a key role in the study of the knotting of minimal sets.
\begin{defn}
For an embedded minimal set $M\hookrightarrow S^3$, the \emph{knot group} is the fundamental group of $E=S^3\setminus M$ and we denote this group by $G(M).$
\end{defn}

Given that $M$ is one-dimensional and compact, we know that $E$ is an open, path connected manifold. Thus, $G(M)$ is defined up to isomorphism independent of the choice of base point, though the choice of base point may influence the precise representation. By an application of the Seifert - van Kampen theorem, just as with classical knots, this knot group is the same up to isomorphism whether we regard $M$ as embedded in $S^3$ or $\mathbb{R}^3$. 

While the computation of this knot group at first seems quite daunting, it can be managed in much the same way as is done with the complement of Antione's Necklace and similar spaces. Based on a neighbourhood presentation $\{S_i\}$ of $M$ as discussed in the previous section, set $E_i = S^3\setminus S_i$ and then we have that $E=\cup_i E_i$ and 
\[
G(M) = \underrightarrow\lim\;\pi_1(E_i)
\]
\noindent where the homomorphisms in the direct limit are induced by the inclusions $E_i\subset E_{i+1}$.

Our first result about this group is a characterisation of unknotted minimal sets in terms of their knot groups, and is an extension of the classical result that an embedded copy of $S^1$ is unknotted if and only if the fundamental group of its complement is $\Z$.

\begin{theorem}\label{unknotfreelimit}
A minimal set $M=\cap S_i\hookrightarrow S^3$ is unknotted if and only if the direct system of groups $\{\pi_1(E_1)\to\pi_1(E_2)\to\cdots\}$ is a system of finitely generated free groups. Hence if $M\hookrightarrow S^3$ is unknotted then $G(M)$ is the direct limit of finitely generated free groups.
\end{theorem}

\begin{proof}
First suppose $M\hookrightarrow S^3$ is unknotted. By definition, the embedded copy of $M$ is $\cap_i S_i$, where each $S_i$ is an unknotted handlebody of genus $n_i$. By an elementary calculation, the complement of any such $S_i$ has free fundamental group of rank $n_i$ and so the first part of the result follows.

Conversely, suppose each $\pi_1(E_i)$ is a finitely generated free group. Then $E_i$ is a regular neighbourhood  of a finite connected graph, and so is also a handlebody. Then $S_i$ and $\overline{E}_i$, the closure of $E_i$, is a Heegaard splitting of $S^3$ of genus the rank of $\pi_1(E_i)$. By Waldhausen's theorem \cite{Wald} any two such splittings of a given genus are isotopic, and hence isotopic to the decomposition of $S^3$ where the handlebody  $S_i$ is unknotted in the sense of Definition \ref{HBunk}.
\end{proof}

Of course this result does not imply that the knot group of an unknotted minimal set is free: for example, the direct limit of copies of $\Z$ under the `times 2' homomorphism $n\mapsto 2n$ has limit $\Z[\frac{1}{2}]$, which is not free. In fact this is precisely what happens in the case of the standard unknotted embedding of the dyadic solenoid. 

When the surfaces $S_i$ are themselves knotted in $S^3$, then the groups $\pi_1(E_i)$ will not be free. Even in the unknotted case, the homomorphisms $\pi_1(E_i)\to \pi_1(E_{i+1})$ arising from the inclusions $E_i\to E_{i+1}$ need not be either surjective, or injective, as we shall see illustrated in examples such as \ref{sol2}  and \ref{TMsimp}  later.

While the knot group $G(M)$ varies considerably with the embedding, the situation is quite different in homology.

\begin{theorem}(Alexander Duality)\label{AlexanderDuality} If $A$ is a closed subset of $S^3$, then (with any coefficients) there is an isomorphism
\[
H_1(S^3\setminus A) \cong \check{H}^1(A),
\]
where $\check{H}$ denotes \v{C}ech cohomology.
\end{theorem}
Thus, we have that $H_1(E)$ is always isomorphic to $\check H^1(M)$ for \emph{any} embedded copy of $M$. Hence, using integer coefficients, the abelianisation of the knot group is always isomorphic with $\check H^1 (M)$, just as the abelianisation of the knot group of a classical knot is always isomorphic to $\mathbb{Z}$.
Similarly, using the functoriality of the terms in the long exact sequence yielding Alexander duality, there is a duality between the maps $\check f_i \colon \check H^1(X_i)\to \check H^1(X_{i+1})$ induced by the bonding maps in a flow expansion and the maps $H_1(E_i)\to H_1(E_{i+1})$ induced by the inclusions in an associated embedding.

\section{Suspensions and density of minimal set homeomorphs}\label{Suspensions}

A transversal or section of a (one-dimensional) minimal set is zero-dimensional. By using the suspension construction, one can recover a minimal set from the return map of the flow to a transversal. Recall that the \emph{suspension} of a homeomorphism $h\colon X \to X$ is the quotient $X\times [0,1]/\sim\,$, where $(x,1)\sim (h(x),0)$, and the \emph{suspension flow} is induced by translation on $[0,1]$. By the results of \cite{AM} there is a homeomorphism from a minimal set $M$ to the suspension of the return map of the flow to any given transversal of the minimal set. If the return map of a minimal set is expansive, then the return map is conjugate to a subshift on a finite alphabet. As we shall see in Section \ref{templates}, suspensions of subshifts are ubiquitous in flows with positive entropy. We will now show that any minimal set in the suspension of a full shift has a dense set of topological copies within the same suspension.

We let $S(n)=\left(\,\mathcal{A}_n^\mathbb{Z},s\right)$ denote the shift on $\mathcal{A}_n=\{0,1,\dots,n-1\}$ and  $\Sigma(n)$ denote the suspension of $S(n)$.

\begin{theorem}\label{density}
Let $M$ be any minimal set of $\Sigma(n)$. The union of all minimal sets of $\Sigma(n)$ homeomorphic with $M$ is a dense subset of $\Sigma(n)$.
\end{theorem}
\begin{proof}
Let $M$ be a given minimal set. With $\mathcal{A}_n^*$ denoting the set of finite words in $\mathcal{A}_n$, let $w$ be any element of $\mathcal{A}_n^*$, and let $k$ be the length of $w$. Identify the subset of $\Sigma(n)$ corresponding to $\mathcal{A}_n^\mathbb{Z}\times \mathbb{Z}$ with $\mathcal{A}_n^\mathbb{Z}$ and let $M_0=M\cap \mathcal{A}_n^\Z$. 

We assume that the symbol $0$ occurs in $M_0$; otherwise, relabel the symbols so that $0$ does occur. Let $\mu-1$ be the maximum number of times that any symbol occurs in $w$. We let $p$ denote the cyclic permutation of $\mathcal{A}_n$ mapping $i$ to $i+1 \mod n$, and we extend $p$ to $\mathcal{A}_n^*$ by applying $p$ to the symbols making up a word and concatenating: $p(v_1\cdots v_\ell)=p(v_1)\cdots p(v_\ell).$ Now consider the following function $\sigma_w \colon \mathcal{A}_n \to \mathcal{A}_n^*$ given by
\[
\sigma_w(0)=0^\mu w 0^\mu, \sigma_w(1)=0^\mu p(w) 0^\mu, \dots, \sigma_w(n-1)=0^\mu p^{n-1}(w) 0^\mu.
\]
For a given word $v=v_0v_1\cdots v_k \in \mathcal{A}_n^*$, let $[v]$ denote the cylinder set
\[
[v]=\{\,(a_i)_{i\in \mathbb{Z}}\in \mathcal{A}_n^\mathbb{Z} \, \colon a_0=v_0, a_1=v_1,\dots, a_k=v_k\,\}.
\]
Letting $K_i=[\sigma_w(i)],$ we let $K$ be the clopen set given by $\cup_i K_i$. Now, for $(x_i)=\cdots x_{-1}\,.\,x_0x_1\cdots \in M_0$, let $\sigma_w\left( (x_i)\right)=\cdots \sigma_w(x_{-1})\,.\,\sigma_w(x_0)\sigma_w(x_1)\cdots$ where the `$.$' separates the terms with negative index from those with non-negative index. We can see that $\sigma_w$ maps $M_0$ injectively to its image.  In fact, the time required for a point of $K$ to return to $K$ under the suspension flow is always $2\mu +k$. Moreover, $\sigma_w$ yields a conjugacy of the return map of the flow to $K$ to the return map of the suspension flow to $M_0$.  Thus, $\sigma_w(M_0)$ is a minimal set of $s$, and the suspension flow on $\sigma_w(M_0)$, which forms a minimal set we denote $\sigma_w(M)$, is conjugate with the suspension flow on $M$ after a rescaling of time by a factor of $2\mu +k$. As the sets in the suspension space of the form $I \,.\, s^\ell([w])$, where $I$ is an interval in $\mathbb{R}$, $w\in \mathcal{A}_n^*$ and $\ell \in \mathbb{Z}$ form a basis for $\Sigma(n)$, we see that the union of the homeomorphic copies of $M$ given by the images $\sigma_w(M)$, $w\in \mathcal{A}_n^*$ is a dense subset of $\Sigma(n)$.
\end{proof}
It should be noted that the maps $\sigma_w \colon \mathcal{A}_n \to \mathcal{A}_n^*$ described above extend to the entire suspension $\Sigma(n)$.

Observe that we cannot expect the minimal sets $\sigma_w(M_0)$  to be homeomorphic or conjugate to $M_0$ as illustrated by the simple example of $M_0$ a fixed point, in which case $\sigma_w(M_0)$ is generally a periodic but not fixed orbit. Of course, the suspension of the orbits in these cases will all be homeomorphic with $S^1$.

Among the minimal sets of  $\Sigma(2)$ are those arising from the Sturmian subshifts of $S(2)$; see, e.g., \cite{BG}. We shall also refer to the minimal sets arising from the suspensions of Sturmians as Sturmians. Such minimal sets are further known as Denjoy minimal sets and occur as minimal sets of $C^1$ flows on the torus. The Sturmian minimal sets were classified by Fokkink \cite{F},\cite{BW}. Two Sturmians are homeomorphic if and only if the continued fraction expansions of the numbers giving the frequency of $1$'s in any of their sequences share a common tail in their continued fraction expansion. Thus, there are uncountably many topologically distinct Sturmian minimal sets, which is important for the results of Section \ref{templates}.

\section{Examples and computations}\label{excon}
We illustrate the computation of the knot group for minimal sets with a few basic examples, concentrating on the inductive step, the computation of the group and homomorphism $\pi_1(E_i)\to \pi_1(E_{i+1})$. For simplicity we describe in detail the case where (without loss of generality) $i=1$ and the handlebody $S_1$ is unknotted; the general case is similar, though more complex in practice. For the cases we need in the later sections for the main theorems, analogues of these examples will suffice.

To compute such an example of $\pi_1(E_2)$ we follow an analogue of the Wirtinger presentation of the group of a (classical) knot, which proceeds in terms of the arcs and crossings of a knot diagram representing it. (See any standard text on knot theory, for example, Lickorish \cite{Lick} chapter 11.)

To compute $\pi_1(E_2)$ and the homomorphism $\pi_1(E_1)\to \pi_1(E_2)$, it suffices to consider the central section of $S_2$, a wedge of circles that can be identified with $X_2$ in the expansion. Take a `knot diagram' of this $X_2$, a 2 dimensional picture of $X_2$ lying in the surface $S_1$ in general position, i.e., with no triple or higher crossings. It consists of connected arcs with under/over crossings (we draw an arc continuing at an over crossing, but breaks to give a new arc at an under crossing), together with the wedge point itself. (See the diagrams later in this section.) We consider all arcs stop at the wedge point and become new arcs beyond it. 

The presentation of the fundamental group of the complement of  $X_2$ is then given by generators in correspondence with the arc components, and relations given by the crossings and the wedge point. We take the following convention:  denote oriented arcs by capitals, $A$, $B$, etc., and the elements of the fundamental group given by loops that wind around them once in accordance with the `right handed corkscrew' rule by the corresponding lower case letters $a$, $b$, etc. This is illustrated by the left hand diagram in  Figure 1 below. The other two diagrams in Figure 1 show the relations at a crossing and at the wedge point for $X_2=S^1\vee S^1$ respectively. (The case of more copies of $S^1$ in $X_2$ is analogous, with, for example supposing the wedge of $n$ circles, a relation that can be written $x_1=x_2^\pm \cdots x_{2n}^\pm$, where the $x_i$ are the names of the in- and out-going arcs for each $S^1$, and the sign is taken according to the direction on each arc. 
The relation from the wedgepoint can also be thought of as saying a loop around all the arcs on one side of the wedgepoint is homotopic to the loop moved over the wedge point to go around all the other arcs.)

The homomorphism $\pi_1(E_1)\to \pi_1(E_2)$ is then read off by identifying the generating loops of $\pi_1(E_1)$ as words in the generating loops of $\pi_1(E_2)$. In the case of $S_1$ unknotted, as considered below, this is straightforward. If $S_1$ is itself knotted (and similarly if this process is carried out repeatedly for higher $S_i$, then the generators of $\pi_1(E_1)$ may be more complicated, but the principle of the calculation is the same.

\begin{figure}[!htb]
	\centering
	\includegraphics[width=120mm]{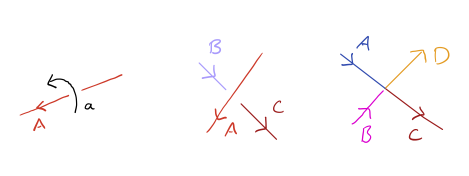}
	\captionsetup{justification=centering}
	\caption{{\sl Generators and relations. In the middle diagram, a crossing, with our convention, gives the relation $a=cab^{-1}$. The right hand diagram, of the wedge point, gives the relation $ab=dc$.}}
\end{figure}

\begin{ex}\label{sol2}{\em 
We begin with the  case of a dyadic solenoid
$$M\ =\ \lim\{\cdots\to S^1\buildrel \times 2\over\longrightarrow S^1\buildrel \times 2\over\longrightarrow S^1\buildrel \times 2\over\longrightarrow S^1\}.$$ 
In Figure 2 we see two different possible embeddings of a regular neighbourhood of the second stage in a regular neighbourhood of the first, in each case the first being a simple solid (unkotted) torus. We denote the orientation of this initial torus by $X$, as shown in the top diagram, and the corresponding generator of $\pi_1(E_1)=\Z$ by $x$. 

\begin{figure}[!htb]
	\centering
	\includegraphics[width=150mm]{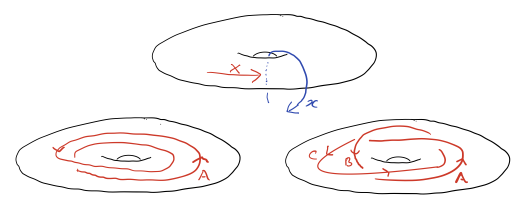}
	\captionsetup{justification=centering}
	\caption{{\sl Two embeddings of stages for the dyadic solenoid.}}
\end{figure}

In the left hand example, the second stage is unknotted, as the first. There is only one arc, called $A$, and so $\pi_1(E_2)=\Z=\langle a\rangle$. (There is one crossing, but it gives only the trivial relation $a=aaa^{-1}$.) The loop $x$ passes around $A$ twice, giving (taking in to account the orientations) $x\mapsto aa$. Iterating this construction, where each handlebody $S_{i+1}$ is unknotted in $S^3$ we obtain the knot group of the unknotted dyadic solenoid simply as
$$\lim\{\Z\buildrel\times 2\over\longrightarrow\Z\buildrel\times 2\over\longrightarrow\Z\to\cdots\}=\Z[{\frac{1}{2}}]\,.$$

In the right hand example, the solid torus $S_2$ inside $S_1$ is knotted, here as a trefoil. While the fundamental group of the complement of $S_1$ is still $\Z=\langle x\rangle$, that of the complement of $S_2$ will be  the knot group of the trefoil. Use of the Wirtinger presentation as described for the figure shown gives
$$\pi_1(E_2)\ =\ \langle a,b,c\,|\, ab=ca=bc\rangle$$
and the image of $x$ is the common word $ab=ca=bc$.

If subsequent solid tori $S_i$ are knotted in non-trivial ways for an infinite number of $i$, then it becomes difficult to say much about the eventual limit group $\lim\pi_1(E_i)$. In Section \ref{surfaceExp} however, we are able to concentrate on cases where for sufficiently large $i$ the $S_i$'s are unknotted within the previous $S_{i-1}$, which makes computation much more feasible.
}\end{ex}

\begin{remark}{\em 
In general, a one dimensional \emph{solenoid} shall mean here any minimal set obtained as the limit of an inverse sequence of circles with orientation-preserving covering maps as bonding maps, and we will only consider the case of solenoids other than $S^1$. 

It is well known that two solenoids $\mS$ and $\mS'$ are homeomorphic if and only if $\check{H}^1(\mS)$ is isomorphic with $\check{H}^1(\mS')$ (Cech cohomology with integer coefficients) as these groups are isomorphic to the respective character groups. Thus, for \emph{any} embeddings of topologically distinct solenoids $\mS$ and $\mS'$, we must have that $S^3\setminus \mS$ is not homeomorphic with  $S^3\setminus \mS'$ as their homology groups  coincide with the respective Cech cohomology groups of the solenoids, by Alexander Duality. It is not immediately clear that this is true for more general minimal sets as $\check{H}^1$ is far from a complete invariant in the general case. 

As can be seen directly or by considering the functoriality of the Alexander Duality previously discussed, we can see that for an \emph{unknotted} solenoid $\mS\hookrightarrow S^3$ the knot group $G(\mS)$ is isomorphic with $\check{H}^1(\mS)$ as abelianisation does not alter the isomorphism type of the fundamental group of the circle. 

Solenoids are the only minimal sets which admit flow expansions with circles as the approximating spaces as any factor map of flows on circles is a covering map. Thus, solenoids are precisely the minimal sets that can be modeled by a sequence of classical knots, and hence their embeddings can be viewed as the `pro-knots.' 

}\end{remark}

\begin{ex}{\em 
Our next example, Figure 3, considers a special Sturmian example given by the closure of the shift orbit in $(\{A,B\}^{\mathbb{Z}},s)$ of the fixed point of the Fibonacci substitution given by $A\mapsto AAB,\ B\mapsto AB$. In both pictures the initial two holed unknotted solid torus $S_1$ is oriented with the positive direction around the holes being anticlockwise; the left hand loop will represent $A$, the right hand $B$ and we denote these $A_1$ and $B_1$ to indicate they are part of the handlebody $S_1$. The corresponding elements of the fundamental group of $E_1$ are $a_1$ and $b_1$, which freely generate. 

The left hand diagram shows an embedded $S_2$ realising the substitution map. It is unknotted in the sense of section \ref{setup}, and in fact can be represented as part of a surface inclusion, as discussed in greater depth in Section \ref{surfaceExp}. It has arcs in the diagram denoted $A_2$, $A_2'$ and $B_2$, with corresponding group elements  $a_2$, $a_2'$ and $b_2$. There is one crossing plus the wedge point. The latter gives the relation $a_2'b=a_2b$ and so $a_2=a_2'$. The crossing now gives the trivial relation. We find that $\pi_1(E_2)$ is the free group $\langle a_2,b_2\rangle$ and the homomorphism induced by the inclusion $E_1\to E_2$ acts as $a_1\mapsto a_2a_2b_2$ and $b_1\mapsto a_2b_2$, i.e., the substitution again, an isomorphism. 

\begin{figure}[!htb]
	\centering
	\includegraphics[width=160mm]{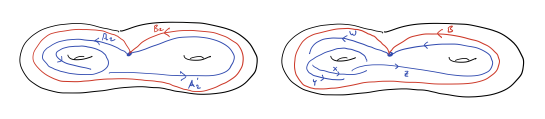}
	\captionsetup{justification=centering}
	\caption{{\sl Two embeddings of stages for the Fibonacci minimal set.}}
\end{figure}

Many forms of complexity could be introduced with more complicated embeddings. In the right hand diagram we  embed $S_2$ by making  the $A_2$ arc perform a trefoil around the left hand hole, as in the second version of the solenoid example; more complex results again could be obtained by, say, making the $A_2$ and $B_2$ arcs link non-trivially as well. For the purpose of illustration though, the right hand diagram, with the arc notation used there, yields
$$\pi_1(E_2)\ =\ \langle w,x,y,x,b\,|\,yw=xy=zx=yz\rangle$$
and the image of $\pi_1(E_1)$ being generated by $yw$ and $zb$, which is an inclusion, but not an isomorphism.
}\end{ex}

\begin{remark}{\em
We note that on passing to homology, hence the abelianisation of these groups, there is exactly one (free abelian) generator for each complete loop in $S_i$: a crossing relation $a=bac^{-1}$ after abelianising gives $b=c$. The image of the resulting homomorphism $H_1(E_i)\to H_1(E_{i+1})$ is then the (transpose) of the transition matrix associated to the substitution $\pi_1(S_{i+1})\to \pi_1(S_i)$, giving a tangible demonstration of the Alexander duality noted before.
}\end{remark}

\begin{ex}\label{TMsimp}{\em 
Our final illustration is one based on the Thue-Morse substitution $A\mapsto AB$, $B\mapsto BA$. Figure 4 illustrates the inclusion of an $S_2$ in the two holed solid torus $S_1$ which  represents this map, and the two stages $S_1$ and $S_2$ are unknotted in the sense of Section \ref{setup}. The iteration of this inclusion does not give the Thue-Morse minimal set in the limit as the substitution does not force the border. Rather one would need to pass to a collared or properised version, which increases the number of wedge summands needed. This simplified inclusion however illustrates the potential for non-injectivity in the homomorphism of fundamental groups of the complements.

\begin{figure}[!htb]
	\centering
	\includegraphics[width=120mm]{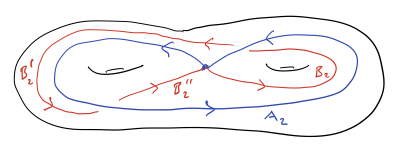}
	\captionsetup{justification=centering}
	\caption{{\sl The simplified Thue-Morse embedding.}}
\end{figure}

The notation is as in the previous examples. There are four arcs in the second stage, two crossings and the wedge point. The latter gives $b_2=b_2''$, while the crossings then both give the relation $b_2a_2=a_2b_2'$. The homomorphism $\pi_1(E_1)\to \pi_1(E_2)$  has the two  generators $a_1, b_1$ take images $a_2b_2'$ and $b_2a_2$, i.e., are equal. Thus the image of the rank 2 group $\pi_1(E_1)$ is rank 1 in $\pi_1(E_2)$, and hence has a kernel.
}\end{ex}

\section{Surface expansions}\label{surfaceExp}

We restrict now to considering minimal sets $M$ in $S^3$ that have an \emph{oriented surface expansion}, Definition \ref{OSE}. Examples in this class, which includes, among others, infinite numbers of each of the Sturmians,  admit a more straightforward description and computation of $G(M)$, and will allow us to prove the main result in Section \ref{templates}.

\begin{defn}\label{OSE}
Suppose $M$ is a minimal set in $S^3$ with  a flow expansion
\[
M \cong \underleftarrow{\lim} \left(X_1 \xleftarrow{f_1} X_2 \xleftarrow{f_2} X_3 
\xleftarrow{f_3} \cdots\right)\,.
\]
We say that this is realised as an \emph{oriented surface expansion} if there is some $N$ such that for all  $i\geqslant N$ there is a sequence of nested oriented  surfaces (with boundary) $\Sigma_i$
\[
\cdots\hookrightarrow\Sigma_i\hookrightarrow\cdots \Sigma_{N+2} \xhookrightarrow{e_{N+1}} \Sigma_{N+1}\xhookrightarrow{e_{N}} \Sigma_{N}\subset S^3 \,,
\]
 where each $\Sigma_i$ is  embedded within the interior of the previous $\Sigma_{i-1}$  with the cross-sectional width of $\Sigma_i$ converging to $0$ as $i\to \infty$, together with a commutative diagram

\[
\xymatrix{
\cdots \,\ar@{^{(}->}[r] & \ar[d]\ar@{^{(}->}[r] \Sigma_{i}
&\cdots \ar@{^{(}->}[r]
& \Sigma_{N+2}\ar[d]\ar@{^{(}->}[r]_{e_{N+1}}
& \Sigma_{N+1}\ar@{^{(}->}[r]_{e_{N}}\ar[d]
& \Sigma_{N}
\ar[d]
\\
\cdots\ar@{->>}[r]
& X_i\ar@{->>}[r]
&\cdots\ar@{->>}[r]
& X_{N+2}\ar@{->>}[r]_{\;\;f_{N+1}}
& X_{N+1}\ar@{->>}[r]_{\;\;f_{N}}
&X_{N}}
\]
\noindent in which the vertical arrows are deformation retractions, and  $M=\cap_i  \Sigma_i$. 

\end{defn}

We also require our surfaces $\Sigma_i$ to be each embedded in the next, and in $S^3$, in a reasonably tame way, so we shall assume from now on that  each of the $\Sigma_{i+1}\to\Sigma_i$ and $\Sigma_i\to S^3$ is PL.

\begin{defn}
    We say a minimal set $M$ in $S^3$ has a \emph{surface embedding} if it is ambiently isotopic to a copy of $M$ that lies on a surface in $S^3$. 
\end{defn}

If $M$ has an oriented surface expansion, then is it isotopic to a set lying on an oriented surface in $S^3$ by virtue of the inclusion $M=\cap \Sigma_i\subset\Sigma_N$. The converse does not seem likely to hold except under further conditions. 

\begin{remark}\label{2to3}{\em 
Given an oriented surface expansion as above, it can be extended to a handlebody presentation as in section \ref{setup} by taking the local cartesian product of each $\Sigma_i$ with a small interval, with the interval for $\Sigma_{i+1}$ embedded as, say, the middle third of the interval for $\Sigma_i$.
}\end{remark}

\begin{defn}
    A minimal set $M$ in $S^3$ has \emph{bounded} oriented surface expansion if it has an oriented surface expansion in which the number of $S^1$'s used in the wedges $X_i$ is bounded for all $i$.
\end{defn}

Note that by telescoping, any bounded expansion $M=\lim X_i$ may be considered an expansion in which all the $X_i$ are wedges of some constant number, $r$ say, of circles. We shall call such an expansion a \emph{constant rank} expansion.

\begin{prop}
The Sturmians all have flow expansions that can be realised as oriented surface expansions.
\end{prop}

\begin{proof}
    These are the minimal sets which have flow expansions  
    $$M \cong \underleftarrow{\lim} \left(X_1 \xleftarrow{f_1} X_2 \xleftarrow{f_2} X_3 
\xleftarrow{f_3} \cdots\right)$$
where each $X_i$ is the wedge of two circles, whose two arcs away from the wedge point we shall denote $0$ and $1$. If we use the properised version (see \cite{CHcom} section 3.1), the maps $f_i\colon X_{i+1}\to X_i$ are all of the form 
$$\sigma_{n_i}\colon 0\mapsto 0\buildrel (n_i+1)\over\ldots 01\qquad\mbox{and}\qquad
1\mapsto 0\buildrel (n_i)\over\ldots 01$$
for some sequence of positive integers $n_i$. 

Let $R$ be the closed surface with boundary given by removing an open disc from a 2-torus. It is depicted in the left hand diagram of Figure 5.
\begin{figure}
	\centering
	\includegraphics[width=160mm]{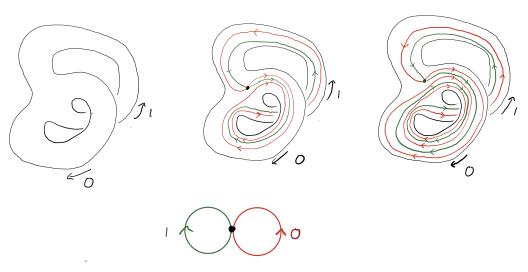}
	\captionsetup{justification=centering}
	\caption{{\sl Embedding of $R$ in $R$ according to the cases $n_i=1$ (centre) and $n_i=2$ (right).}}
\end{figure}
Then $R$ deformation retracts to $X_i=S^1\vee S^1$. It suffices to show that for each $n_i$ we can embed $R$ in itself, realising the homomorphism $\sigma_{n_i}$ in $\pi_1(R)$. Figure 5 shows this done for the cases $n_i=1,2$, and the general case is similar: the red track spirals round the $0$ arc $n_i+1$ times before going to the $1$ arc (once), the green track spirals round the $0$ arc $n_i$ times, in between the red tracks before finishing on a loop of the $1$ arc.
\end{proof}

\begin{remark}{\em
  As the Sturmian homomorphisms $\sigma_{n_i}$ all induce isomorphisms of $F_2$, the free group on two letters, a simple application of Theorem \ref{unknotfreelimit} shows that the knot group of \emph{any} unknotted Sturmian minimal set in $S^3$ is $F_2$, so the knot group does not distinguish between distinct unknotted Sturmians. They can of course easily be distinguished using some analogue of the homology core \cite{CHcore}, which remembers the slope of the Sturmian modulo $Gl(n,\Z)$, as in Fokkink's results \cite{F},\cite{BW}.
}\end{remark}

Recall the classification theorem of closed oriented surfaces with boundary: each such surface is classified by the number of boundary components, $b$ say, and the genus for the resulting surface without boundary given by gluing in $b$ discs to the boundary components. We say such a surface to be \emph{of type} $(b,g)$. Thus, for example, $(2,1)$ is the regular torus with two open discs removed, while $(1,0)$ is the 2-sphere with one open disc removed, i.e., is a closed disc. Moreover, so long as $b>0$, all such surfaces are homotopy equivalent to a wedge of some finite (possibly zero) number of circles. Elementary calculation computes the homology of these objects.

\begin{lemma}
    Let $\Sigma$ be a surface of  type $(b,g)$, with $b\geqslant 1$. Then its first homology group $H_1(\Sigma;\Z)$  is free abelian of rank $2g+b-1$, which is also the number of circles in the wedge to which $\Sigma$ is homotopy equivalent. Hence, for any given rank $r$, there are only a finite number of surfaces which have this as the rank of their $H_1$.
\end{lemma}

\begin{prop}
Suppose $M$ is a minimal set in $S^3$ with a bounded oriented surface expansion. Then there is an oriented surface expansion $(\Sigma_i,e_i)$ such that, for all $n$ sufficiently large, the surfaces $\Sigma_i$ are all of the same $(b,g)$ type.
\end{prop}

\begin{proof}
The bounded surface expansion models a bounded expansion $\{X_i\}$ which, as noted above, we may assume to be a constant rank expansion. As there are only a finite number of distinct $(b,g)$-types of surface in the set of $\Sigma_i$'s, at least one of these must occur infinitely often. After potentially further telescoping, we can assume all the $\Sigma_i$ are of this type. 
\end{proof}

\begin{prop} \label{cute}
Suppose $R$ is a compact, connected oriented surface with boundary of some type $(b,g)$. Suppose $f\colon R\to R$ is an embedding, with image contained in the interior of the target $R$. Then the induced homomorphism $f_*\colon \pi_1(R)\to\pi_1(R)$ is an isomorphism.
\end{prop}

\begin{proof}
    We have $\pi_1(R)=F$, a free group of some rank $q$. Suppose $f_*$ is not an isomorphism. If it is onto then it would be an isomorphism by the co-Hopfian property of free groups, so it cannot be onto. Then form an inverse system 
    $$\cdots\to R\buildrel f\over\longrightarrow R\buildrel f\over\longrightarrow R\to\cdots\to R\,. $$
   The inverse limit of these spaces is a continuum inside the initial surface $R$. However, the inverse system of fundamental groups fails to be Mittag Leffler, and so contradicts the result that  continua embedded in such surfaces are movable \cite{Btext} VII \S7, \cite{Kras}.
\end{proof}

\begin{prop}\label{surfaceprop}
Suppose $Q$ and $R$ are both closed oriented surfaces of the same type $(b,g)$, with $b\geqslant 1$, and that $Q$ is a subsurface of $R$, embedded in its interior $$Q\buildrel e\over\hookrightarrow\mbox{int}(R)\subset R\,.$$ Then the (closure of) the complement of $Q$ in $R$ is a disjoint set of $b$ cylinders, $S^1\times I$.

\end{prop}

\begin{proof}
 Let $Z$ be the closure of the complement $R\setminus Q$, and $B=Q\cap Z$. Then $B$ consists of $b$ disjoint copies of $S^1$. Consider the Mayer-Vietoris exact sequence of the decomposition $R=Q\cup Z$, noting $H_2$ of each of the spaces concerned is trivial.
$$0\to H_1(B)\buildrel i_1\over \to H_1(Q)\oplus H_1(Z)\buildrel j_1\over \to H_1(R)\buildrel \Delta\over \longrightarrow H_0(B)\buildrel i_0\over \to H_0(Q)\oplus H_0(Z)\buildrel j_0\over \to H_0(R)\to 0\,.$$
We know $Q$ and $R$ are connected, and $B$ has $b$ components. Moreover, each of $Q$ and $R$ have the homotopy type of the wedge of $q=2g+b-1$ circles. So the sequence becomes 
$$0\to \Z^b\buildrel i_1\over \longrightarrow \Z^q\oplus H_1(Z)\buildrel j_1\over \longrightarrow \Z^q\buildrel \Delta\over \longrightarrow \Z^b\buildrel i_0\over \longrightarrow \Z\oplus H_0(Z)\buildrel j_0\over \longrightarrow  \Z\to 0\,.$$
By Proposition \ref{cute}, the homomorphism $e_*\colon\pi_1(Q)\to \pi_1(R) $ is an isomorphism, and hence $e$ will also induce an isomorphism in $H_1$. Thus $j_1$ is onto and so $\Delta$ is the zero map. Counting ranks we obtain $H_1(Z)=\Z^b$. Similarly, $H_0(Z)=\Z^b$. Thus $Z$ has $b$ connected components, and so is the disjoint union of $b$ surfaces, possibly  with boundary; moreover each component of $Z$ must have at least one boundary component since no closed surface without boundary can embed in $R$, a closed surface with non-empty boundary. 

\begin{lemma}\label{cpts}
No two components of $B$ can lie in the same component of $Z$, hence every component of $Z$ contains exactly one component of $B$.
\end{lemma}

\begin{proof}
This follows from the injectivity of the homomorphism $i_0$ in the Mayer Vietoris sequence, and the observation above that  $Z$ has exactly $b$ components.
\end{proof}

Choosing a base point $x_0\in Q$, and paths from $x_0$ to each boundary component of $Q$,  we may consider each such boundary component as an element of $\pi_1(Q,x_0)$ by traversing the path from $x_0$  to that component, running round the boundary once, and then back along the path to $x_0$.
\begin{lemma}
Each boundary component of $Q$ represents a non-trivial element of $\pi_1(Q)$.
\end{lemma}
\begin{proof}
First suppose $Q$ is a surface of type $(b,g)$ with $g>0$. As $Q$ as a $g$-holed torus with $b$ discs removed, and the $g$-holed torus can be constructed as a $4g$-gon with identifications around the edge according to the usual commutator expression,  removing one disc gives the boundary component that the commutator represents as an element of $\pi_1$. Each subsequent disc removed increases a loop in the resulting wedge of circles, and so represents a new (non-trivial) element of the fundamental group. The case of $g=0$, that is the 2-sphere with $b$ holes removed, follows by direct calculation. (Note that if $g=0$ then $b\geqslant 2$ since $Q$ is not simply a disc.)
\end{proof}

Let $\beta$ be a single boundary component of $B$, and $P_\beta$ its path component in $Z$. We know that $P_\beta$ is a closed surface with at least one boundary component.

\begin{lemma}
Each $P_\beta$ is homeomorphic to an annulus, i.e., a surface of type $(2,0)$, with one boundary component a boundary component in $R$ and the other boundary component a boundary component of $Q$.
\end{lemma}

\begin{proof}
Consider $\beta$ as an element of $\pi_1(Q)$ as above. As $e_*\colon\pi_1(Q)\to\pi_1(R)$ is an isomorphism, $P_\beta$ cannot be a disc, a surface of type $(1,0)$, hence it is a surface of type $(b',g')$ with either $b'>1$ or $g'>0$. Thus as $H_1(P_\beta)=\Z^{2g'+b'-1}$ this homology group is of rank at least 1. As $Z$ has $b$ components, and $H_1(Z)=\oplus_\beta H_1(P_\beta)=\Z^b$, we must have $H_1(P_\beta)=\Z$, i.e., each $P_\beta$ is of type $(2,0)$, an annulus. By Lemma \ref{cpts}, each $P_\beta$ has one boundary circle in the boundary of $Q$, and the other must then be a boundary circle of $R$.
\end{proof}

This completes the proof of Proposition \ref{surfaceprop}.
\end{proof}

\begin{cor}
    If $Q\subset R$ are closed surfaces as in Proposition \ref{surfaceprop}, then there is a deformation retract of $R$ onto $Q$, i.e., there is a homotopy $H\colon R\times I\to R$ such that $H(r,1)=r$ and $H(r,0)\in Q$ for all $r\in R$,  $H(q,t)=q$ for all $q\in Q$ and $t\in I$
\end{cor}

\begin{proof}
    This follows from Proposition \ref{surfaceprop} and the deformation of $S^1\times I$ onto $S^1=\{(x,0)\in S^1\times I$ by $H((x,s),t)=(x,st)$.
\end{proof}

By extending this deformation over a small annulus neighbourhood in $Q$ of each boundary component, we obtain

\begin{cor}\label{Isoext}
    There is an isotopy of $R$ onto $Q$, i.e., a map $L\colon R\times I\to R$ such that $L(-, 1)$ is the identity on $R$, $L(-,0)$ is a homeomorphism $R\to Q$, and each $L(-,t)$ is a homeomorphism onto its image.
\end{cor}

\begin{cor}\label{isocor}
Suppose $M\subset S^3$ is a minimal set which has an constant oriented surface expansion $\{(\Sigma_i,e_i)\}$. Suppose the handlebodies $S_i$ are the thickenings of the $\Sigma_i$ as in Remark \ref{2to3}. Denote by $E$ the complement of $M$ in $S^3$ and by $E_i$ the complement of $S_i$, Then $G(M)$,  the knot group of $M$, is isomorphic to $\pi_1(E_N)$.
\end{cor}

\begin{proof}It suffices to show that the induced homomorphism $\pi_1(E_i)\to\pi_1(E_{i+1})$ is an isomorphism for all $i\geqslant N$.

Recall the construction of $S_i$ from $\Sigma_i$. The space $S_i$ is the space of the normal bundle  to $\Sigma_i$ with fibre a small closed interval which we shall parameterise as $[-1,1]$; it needs to be small enough so that $S_i$ is homeomorphic to $\Sigma_i\times [-1,1]$. Within this space, $S_{i+1}$ is the subspace corresponding to points $(s,t)$ where $s\in \Sigma_{i+1}\subset \Sigma_i$ and $t\in [-\frac{1}{3},\frac{1}{3}]$. Let $\breve S_i\subset S_i$ denote the subspace of points $(s,t)$ where $s\in\Sigma_i$ and $t\in [-\frac{1}{3},\frac{1}{3}]$. 

Let $\breve E_i$ be the complement of $\breve S_i$. As $E_i$ and $\breve E_i$ are homeomorphic, they have the same fundamental group. By Proposition \ref{surfaceprop} the (closure of) the subspace $\breve S_i\setminus S_{i+1}$ consists of the disjoint union of locally thickened annuli, which we can parameterise as spaces $S^1\times [0,1]\times [-\frac{1}{3},\frac{1}{3}]$. Here we shall take the middle parameter in $[0,1]$ to be arranged so that the 1 end is on the boundary of $\Sigma_{i+1}$, and the 0 end on the boundary of $\Sigma_i$. The space $E_{i+1}$ can then be considered as the space $\breve E_i$ with these closed, thickened annuli glued on.

We claim that $\breve E_i$ is a homotopy retract of $E_{i+1}$, i.e., there is a homotopy $K\colon E_{i+1}\times I\to E_{i+1}$ such that $K(y,1)=y$ and $K(y,0)\in \breve E_i$ for all $y\in E_{i+1}$,  and $K(x,t)=x$ for all $x\in \breve E_i$ and $t\in I$. To define this, it suffices to define $K$ on points in the disjoint union of thickened annuli $\breve S_i\setminus S_{i+1}$, each of which we have parameterised as $S^1\times [0,1]\times [-\frac{1}{3},\frac{1}{3}]$. Then $K$ on such an annulus is given by $K((z,v,t),r)=(z,rv,t)$, i.e., the homotopy collapses the annulus down to the end that is part of the boundary of $\breve S_i$.

Hence there is a homotopy equivalence between $\breve E_i$ and $E_{i+1}$ and the result follows.

\end{proof}

Again, as in Corollary \ref{Isoext}, this may be extended to homeomorphisms of the spaces concerned
\begin{cor}\label{isorem}
Suppose $M\subset S^3$ is a minimal set which has an constant oriented surface expansion $\{(\Sigma_i,e_i)\}$ and corresponding complements $E_i$ to their thickenings, as before. Then there are homeomorphisms 
     $E_N\cong E_{N+1}\cong E_{N+2}\cong\cdots\cong E$.
\end{cor}

\begin{cor}
If $M$ is a minimal set in $S^3$ that has a bounded, oriented surface expansion, then its knot group is finitely generated.
\end{cor}

\begin{proof}
By Corollary \ref{isocor} the knot group of $M$ is given by the fundamental group of the complement of some finite stage $E_N$. As $E_N$ is $S^3$ with a finite CW complex removed, the result follows.
\end{proof}

\begin{remark}{\em
The requirement in  Corollary  \ref{isocor} that the minimal set $M$ has an oriented surface expansion is strictly necessary. For example, to take the Sturmian examples considered earlier, it is easy to construct  handlebody presentations $M=\cap_iS_i$ in which the embeddings $S_{i+1}\to S_i$ at each stage braid the images of the $S^1\vee S^1$; we saw an example of this in the right hand diagram of Figure 3. While that maps $S_{i+1}\to S_i$ still induce the same isomorphism in the fundamental group, the resulting homomorphisms $\pi_1(E_i)\to \pi_1(E_{i+1})$ are no longer isomorphisms.
}\end{remark}

\begin{remark}\label{Handel}
{\em 
Handel \cite{H} shows that any one-dimensional minimal set $M$ of a flow in $S^3$ that occurs in isolation is necessarily a surface minimal set; i.e., it occurs as a minimal set of a flow on a surface. This was relevant to the then open Seifert conjecture, for which Schweitzer \cite{S} had constructed a counter-example with an isolated minimal set supporting a $C^1$ flow. In fact, the example constructed by Schweitzer is an unknotted Sturmian minimal set embedded in much the same way as in Figure 5. In general, using Schwietzer's technique it is straightforward to embed an unknotted minimal set of an oriented surface flow as an isolated minimal set of a $C^1$ flow on $S^3$. While the aperiodic, one-dimensional minimal sets on a surface have limited differentiablity, these minimal sets occur intertwined in larger invariant sets of unlimited differentiability as we shall see in Section \ref{templates}. While Kuperberg \cite{K} has settled the Seifert conjecture for smooth flows, it would be interesting to determine whether there is any restriction on the knotting for an isolated one-dimensional minimal set of a $C^1$ flow on $S^3$.
}\end{remark}

\section{Templates and the positive entropy case}\label{templates}

In \cite{BW1},\cite{BW2} Birman and Williams introduced templates to model the embedding of suspensions of shifts in flows in $S^3$. A comprehensive treatment of templates and periodic orbits can be found in \cite{GHS}. The basic Lorenz template, modelled on the Lorenz attractor, is formed as indicated in Figure 6. Such a template supports a semi-flow (an action of $[0,\infty)$) formed by the suspension of the one-sided shift $\left(\{0,1\}^{\mathbb{N}_0},s\right)$. Many orbits do not remain on the template, but the template supports the entire semi-flow of the suspension of the shift. Such templates are quotients of the suspension of horseshoes and similar maps, where the stable manifolds are identified to points. Thus, the semi-flow on the template is not an entirely faithful representation of the flow on an invariant set of a flow in $S^3$, but it is constructed to faithfully represent the knotting and linking of orbits in the invariant set. For example, in the case of a Sturmian minimal set, which has one pair of orbits that approach each other asymptotically in forward time, there will be one pair of orbits which are identified in the template after some point, but the orbits will be nonetheless distinguished in the template as we explain below. All other orbits of a Sturmian are faithfully represented.

 \begin{figure}[!htb]
	\centering
	\includegraphics[width=80mm]{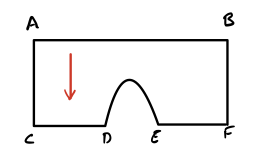}
	\captionsetup{justification=centering}
	\caption{{\sl The basic template. The edge $CD$ is stretched to join $AB$, as is $EF$. The flow lines are vertical with direction as indicated by the red arrow. When a flowline passes into the arch, between $D$ and $E$, it disappears from the template.}}
\end{figure}

 \begin{figure}[!htb]
	\centering
	\includegraphics[width=80mm]{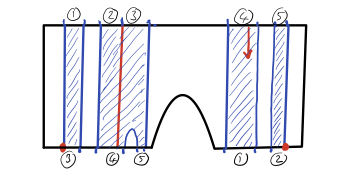}
	\captionsetup{justification=centering}
	\caption{{\sl Neighbourhood scheme of the first stage of the image of the Fibonacci minimal set in the template.}}
\end{figure}

In Figure 7 we see a schematic neighbourhood of the first stage of the image in the template of the minimal set $M_{Fib}$ arising from the Fibonacci substitution $0\mapsto 010$, $1\mapsto 01$. As before, the flow in the template is vertically downwards, and the paths that hit the left section represent 0, while those on the right give 1's. The numbered intervals of the neighbourhoods at the bottom are identified with the corresponding number at the top. (There is a magnification arising from the stretching factor which is suppressed in the diagram.)

This minimal set has two singular orbits, corresponding to the bi-infinite words
$$\begin{array}{rl}
\cdots 01001\!\!\!\!& .\ 010010\cdots\\
\cdots 01010\!\!\!\!& .\ 010010\cdots
\end{array}$$
and in the template the right hand half infinite sections of these two orbits become identified. We denote by $\widetilde{M_{Fib}}$ this image of $M_{Fib}$. In the figure the image of this double orbit from the decimal point onwards is highlighted in red, traveling from the junction of boundaries of the strips marked 2 and 3 at the top, through the 0 arc, to a point in the middle of the interval 4; the path continues from the 4 at the top, describing the 1 arc, and so on. The pre-image of the decimal point at the top is a two-fold branch point: on the bottom of the template the pre-image consists of a point in the neighbourhood strip 2 (describing a 1 arc) and in the strip 3 (describing a 0 arc). 

\begin{aside}
    \emph{For those familiar with the work of Anderson and Putnam \cite{AP} on the spaces associated to substitutions, $\widetilde{M_{Fib}}$ is precisely the space obtained for this substitution using the \emph{uncollared} complex. This is a strict quotient of $M_{Fib}$ since the substitution $0\mapsto 010$, $1\mapsto 01$ does not force the border. In the earlier construction, illustrated in Figure 5, we used a properised version of the substitution so as to ensure that the limit space was the genuine minimal set.}
\end{aside}

Other Sturmian minimal sets have similar associated diagrams, though with more strips involved. The scheme for the substitution $0\mapsto 0\buildrel(r)\over\cdots 010$, $1\mapsto0\buildrel(r)\over\cdots 01$ is illustrated in Figure 8.

 \begin{figure}[!htb]
	\centering
	\includegraphics[width=80mm]{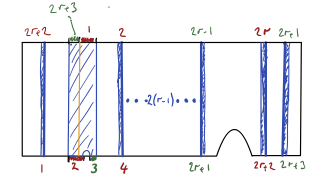}
	\captionsetup{justification=centering}
	\caption{{\sl Diagram for the template representing the first stage for Sturmian map $0\mapsto 0\buildrel(r)\over\cdots 010$, $1\mapsto0\buildrel(r)\over\cdots 01$. The numbers at the top and the bottom indicate where the strips match up; the red represent parts of the loop corresponding to the $0\buildrel(r)\over\cdots 010$ word, the green to the $0\buildrel(r)\over\cdots 01$ word.}}
\end{figure}

\bigskip This establishes the first stage of our Sturmian as having a surface neighbourhood in the template. Subsequent finer neighbourhoods (whether Fibonacci or other Sturmian words) will thus also lie in this surface, embedded analogously. 

 \begin{figure}[!htb]
	\centering
	\includegraphics[width=80mm]{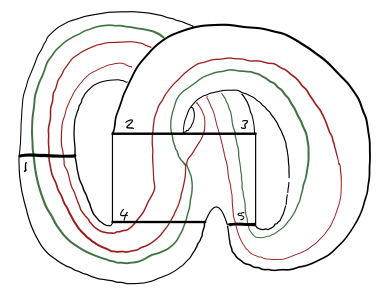}
	\captionsetup{justification=centering}
	\caption{{\sl Embedding of the next stage of the image of the Fibonacci minimal set in the template.}}
\end{figure}

As an example, in Figure 9 we see the strips of neighbourhood shown in Figure 7 glued together according to their numbered joins there. In it is a graph whose thickening gives a neighbourhood of the next stage of the Fibonacci space: the red loop indicates the double substitution of 0, i.e., $0\mapsto 010\mapsto 01001010$, and the green that for 1, namely $1\mapsto 01\mapsto 01001$. As can be seen the pattern of neighbourhoods in the middle section is replicating the overall pattern in the template in Figure 7, reflecting the self-similarity of $M_{Fib}$ under substitution.

\bigskip For any given Sturmian, denote by $E'_1$ the complement of this first stage in $S^3$ and by $E'_n$ the complement of subsequent ones. Then $E'=\cup_nE_n'$ is the complement of the Sturmian minimal set with the two singular orbits identified for the positive half line.

As all the Sturmian words induce isomorphisms in $\pi_1(-)$ of the corresponding wedges of circles, the results of Section \ref{surfaceExp}, in particuar Corollary \ref{isorem}, gives a homeomorphism between $E'_1$ and $E'_n$ for all $n$, and hence between $E_1'$ and $E'$. However, if $E'(r)$ denotes the space $E'$ with the two singular lines identified only from point $r>0$ onwards, (so $E'=E'(0)$), then $E$, the complement of the true minimal set of the Sturmian system, is the limit of the $E'(r)$ as $r\to\infty$. As $E'(r)\cong E'(r+1)$, we obtain

\begin{prop}
    Suppose $M$ is a Sturmian minimal set in $S^3$, and realised via a template, as above. Then the knot type of $M$ is given by the knot type of the first stage of its quotient, i.e., by its corresponding $E_1'$ as just described.
\end{prop}

\begin{cor}
Let $M$ be a given Sturmian minimal set. In the flow represented by the template, there are an infinite number of minimal sets homeomorphic to $M$, each with distinct knot type.
\end{cor}

\begin{proof}
    The first stage $E_1'$ of any realisation of $M$ in the template is the embedding of a space with homotopy type a wedge of two circles. By Theorem \ref{density} there are an infinite number of possible examples of a given $M$, with different words (of increasing length over the different examples) describing the examples.  The maps $\sigma_v$ will map a Sturmian minimal set to a homeomorph preserving the combinatorics of the return map so that the homeomorph will lie in a surface but with more crossings, depending on the word $v.$ 

    We consider, say, the 0 loop. For each example its image in the template for the first stage will wind increasing numbers of times around. The argument of Franks and Williams \cite{FW} section 4 now applies: the knot type of a loop in the template has Seifert surface whose genus is bounded below by a positive polynomal expression in the crossing numbers arising in the template. Over our different manifestations of $M$, the 0 loop will eventually describe an infinite number of distinct knots (as will the 1 loop). 
\end{proof}

\begin{remark}{\em
While this establishes that for a given Sturmian $M$ there are an infinite number of distinct knot classes, the Franks and Williams argument adapted above only shows that as we increase the length of the 0 loop we will eventually achieve a new class at some point. It does not tell us how long we must wait, or which are the distinct ones, so it is essentially a non-constructive proof. On the other hand, by Corollary \ref{isocor} we also have a direct ability to compute the knot groups of individual examples explicitly by a one stage computation, just as in the examples in section \ref{excon}, and specify distinct copies of $M$. It should be noted though as the knot group of a (classical) knot is not a complete invariant, this cannot be relied on to catch all the distinct classes.
}\end{remark}

\noindent{\em Proof of Theorem \ref{ambitious}.} By the results of \cite{FW}, any $C^2$ flow on $S^3$ with a compact invariant set with positive entropy has an invariant set modelled by the basic template described above. As noted at the end of Section \ref{Suspensions}, there are uncountably many homeomorphism classes among the Sturmian minimal sets. The theorem follows. \hfill$\square$

\end{document}